\documentclass{amsart}

\usepackage{amssymb}
\usepackage{amsmath}
\usepackage{array}
\usepackage{enumerate}
\usepackage{multirow}
\usepackage{float}

\newcommand{\Pf}{{\em Proof}. }
\newcommand{\EPf}{\hfill$\square$}

\newcommand\fieldsetc{\mathbb}

\newcommand{\R}{\fieldsetc{R}}

\newtheorem{thm}{Theorem}[section]
\newtheorem{cor}{Corollary}[section]
\newtheorem{prop}{Proposition}[section]
\newtheorem{lem}{Lemma}[section]

\theoremstyle{remark}
\newtheorem{rmk}{Remark}[section]

\newtheorem{quest}{Question}[section]

\newcommand{\D}{\mbox{$\mathcal D$}}
\newcommand{\N}{\mbox{$\mathcal N$}}

\title{The $\kappa$-nullity of Riemannian manifolds and\\their
  splitting tensors}

\author{Claudio Gorodski}\thanks{The first named author
  has been partially supported by the grant number 302882/2017-0
  from the \emph{National Council for
  Scientific and Technological Development} (CNPq, Brazil)
  and the project number 2016/23746-6 from the \emph{S\~ao Paulo
  Research Foundation}
  (Fapesp, Brazil).}
\address{Instituto de Matem\'atica e Estat\'\i stica, Universidade de
  S\~ao Paulo, Rua do Mat\~ao, 1010, S\~ao Pulo, SP 05508-090, Brazil.}
\email{gorodski@ime.usp.br}

\author{Felippe Guimar\~aes}\thanks{The second named author
has been supported by the post-doctoral grant 2019/19494-0
from the \emph{S\~ao Paulo Research Foundation}
  (Fapesp, Brazil).}
\address{Instituto de Matem\'atica e Estat\'\i stica, Universidade de
  S\~ao Paulo, Rua do Mat\~ao, 1010, S\~ao Pulo, SP 05508-090, Brazil.}
\email{felippe@impa.br}

\date{\today}


\begin{document}

	\maketitle

        \begin{abstract}
  We consider Riemannian $n$-manifolds $M$ with nontrivial $\kappa$-nullity
  ``distribution'' of the curvature tensor~$R$, namely,
  the variable rank distribution of tangent subspaces to $M$ where $R$
  coincides with the curvature
  tensor of a space of constant curvature~$\kappa$
  ($\kappa\in\R$) is nontrivial. We obtain classification
  theorems under diferent additional assumptions, in terms of
  low nullity/conullity, controlled scalar curvature
  or existence of quotients of finite volume.
  We prove new results, but also revisit previous ones.
  \end{abstract}

  \section{Introduction}

Several important classes of Riemannian manifolds $M$ are defined
by imposing a certain condition on its Riemann curvature tensor $R$,
such as spaces of constant curvature, Einstein manifolds,
locally symmetric spaces, etc. In a somehow different sense, 
it is a stimulating problem to define a class of Riemannian manifolds
by imposing a certain
\emph{form} on their curvature tensors. More specifically,
let $T$ be an algebraic curvature tensor. A Riemannian manifold $M$ is said
to be \emph{modelled} on~$T$ if its
curvature tensor is, at each point, orthogonally equivalent to~$T$.
Here the size of the orbit of $T$ under the action of the orthogonal group
plays a certain role; for instance, curvature tensors
of spaces of constant curvature are fixed points of that action,
and in this case a manifold modelled on $T$ will obviously also have
constant curvature. On the other hand, if $T$ is only required to be the
curvature tensor of a homogeneous Riemannian manifold $\bar M$,
there are continuous families
of examples of complete irreducible Riemannian manifolds $M$ modelled on $T$
which are not locally isometric to~$\bar M$ (see e.g.~\cite{KTV}
for examples and a discussion of related results, which originate
from a question of Gromov).  

If a Riemannian manifold $M$ is modelled on an algebraic
curvature tensor $T$, then clearly it is also
\emph{curvature homogeneous}, in the sense
that the curvature tensors at any two of its points are orthogonally 
equivalent. The totality of curvature homogeneous manifolds
(for varying~$T$) obviously include locally homogeneous spaces,
but contains strictly more manifolds. The first examples 
were constructed
by Takagi~\cite{takagi2} and Sekigawa~\cite{sekigawa}, 
in response to a question by Singer
(these were later generalized, see~\cite{BKV} for the full range of
generalizations).

In a different vein, a Riemannian manifold is called
\emph{semi-symmetric} if its curvature tensor is, at each point,
orthogonally equivalent to the curvature tensor of a symmetric space; the
symmetric space may depend on the point (in particular
a curvature homogeneous semi-symmetric space is a
Riemannian manifold modelled on the curvature tensor of a
fixed symmetric space). In 1968, Nomizu conjectured that 
every complete irreducible semi-symmetric 
space  of dimension greater than or equal to three
would be locally symmetric.
His conjecture was refuted by Takagi~\cite{takagi} and
Sekigawa~\cite{sekigawa2}, who constructed
counterexamples (see~\cite{BKV} for further developments).
The complete classification of
semi-symmetric spaces is the work of Z. I. Szab\'o~\cite{szabo}. 
On the other hand, Florit and Ziller
have shown that the Nomizu conjecture holds for manifolds of finite
volume~\cite{FZ}.

It is remarkable what all of the examples above (and others)
have in common, namely, their curvature tensor has a large nullity. 
This leads us to the 
class of Riemannian manifolds
that we consider herein; loosely speaking, we say
a Riemannian manifold has \emph{non-trivial
  $\kappa$-nullity}, where $\kappa\in\R$, if the variable
rank tangent distribution where its curvature tensor behaves like
that of a space of constant curvature $\kappa$ is non-trivial
(as an extrinsic counterpart to the above examples,
recall that, owing to the Beez-Killing theorem,
a locally deformable hypersurface in a space form of curvature~$\kappa$,
without isotropic points,
has precisely two nonzero principal curvatures at each point, and hence
has a $\kappa$-nullity distribution of codimension~$2$).

The idea of nullity
was introduced in case $\kappa=0$
by Chern and Kuiper in~\cite{ChernKuiperNull}, and for
general $\kappa$ by Otsuki~\cite{O}, and later
reformulated and studied by different authors
(see e.g.~\cite{GrayNull,maltzCompl} and, for more recent
work, \cite{FZ,DOV,DOV2} and the references therein). 
Each sign of $\kappa$ (positive, negative or zero) yields
results of a different flavor. In this paper we consider 
the three cases, and note that the concept of
nullity has connections
with diverse areas such as Sasakian manifolds, solvmanifolds,
and non-holonomic geometry. Our main tool is the so called
\emph{splitting tensor} (cf.~section~\ref{prelim}). 
We prove new results, but we also aim 
to extend, unify and simplify existing results in the literature. 

More precisely, let $M$ be a connected 
Riemannian manifold, and consider the curvature
tensor $R$ of its Levi-Civit\`a connection $\nabla$ with the sign
convention
\[ R(X,Y)Z=\nabla_X\nabla_YZ-\nabla_Y\nabla_XZ-\nabla_{[X,Y]}Z, \]
for vector fields $X$, $Y$, $Z\in\Gamma(TM)$. 
For $\kappa\in\R$, the
\emph{$\kappa$-nullitty distribution} of $M$ is the
variable rank
distribution
$\N_\kappa$ on $M$ defined for each $p\in M$ by
\[ \N_\kappa|_p=\{z\in T_pM:R_p(x,y)z=-\kappa(\langle x,z\rangle_p y-
  \langle y,z\rangle_p x)\quad\mbox{for all $x$, $y\in T_pM$}\}. \]
The number $\nu_\kappa(p):=\dim \N_\kappa|_p$ is called the
\emph{index of $\kappa$-nullity} at~$p$.

In case $\kappa=0$ we obtain trivial examples of manifolds
with positive $\nu_0$ simply by taking a Riemannian product
with an Euclidean space, but similar product
examples do not occur if $\kappa\neq0$.  
It is easily seen that
$\nu_\kappa(p)$ is nonzero for at most one value of $\kappa$.
For general~$M$, $\nu_\kappa$ is nonnecessarily constant if nonzero,
but it is an upper semicontinuous function, so there is an open
and dense set of $M$ where $\nu_\kappa$ is locally constant,
and there is
an open subset $\Omega$ of $M$ where $\nu_\kappa$ attains its minimum
value. It is known that $\N_\kappa$ is an autoparallel distribution
on any open set where $\nu_\kappa$ is locally constant  and,
in case $M$ is a complete Riemannian manifold,
its leaves in $\Omega$ are \emph{complete} totally
geodesic submanifolds of constant curvature $\kappa$~\cite{maltzCompl}.

We call the orthogonal complement of $\mathcal N_\kappa$ the 
\emph{$\kappa$-conullity distribution} of $M$, and its
dimension at a point~$p\in M$ the \emph{index of $\kappa$-conullity} 
at~$p$, or simply, the $\kappa$-conullity at~$p$. For obvious
reasons, the minimal nonzero value of the $\kappa$-conullity is~$2$.
Riemannian manifolds with $0$-conullity at most $2$
have pointwise the
curvature tensor of an isometric product of Euclidean space 
with a surface with constant curvature and hence are semi-symmetric.
Conversely, a complete irreducible
semi-symmetric space is either locally symmetric or
has $0$-conullity at most $2$ in an open and dense subset~\cite{szabo}. 

In our study we apply a homothety and assume that $\kappa$ is equal to
$+1$, $-1$ or~$0$. Generally speaking,
the results below give characterizations/classifications
of manifolds in terms of low conullity/nullity, controlled scalar
curvature and/or existence of quotients of finite volume.
Some terminology: in general we shall say an $n$-manifold has
\emph{minimal} $\kappa$-nullity~$d$ (resp.~\emph{maximal}
$\kappa$-conullity $n-d$) to mean that $\nu_\kappa\geq d$
everywhere and the equality holds at some point. 

\subsection{Results with $\kappa=+1$}

The following theorem gives a lot of rigidity in the case
of constant $(+1)$-conullity~$2$ and constant scalar curvature.
It should be compared with the examples constructed
in~\cite{schmidt-wolfson-cvc1} of certain inhomogeneous 
conformal deformations of left-invariant metrics on~$SU(2)$,
$\widetilde{SL(2,\R)}$, and $Nil^3$,
which posses $(+1)$-conullity $2$ and nonconstant scalar curvature.

\begin{thm}\label{kappa1}
  Let $M$ be a simply-connected complete Riemannian $n$-manifold with
  constant
  $(+1)$-conullity equal to~$2$, and constant scalar
  curvature. Then $M$ is a
  $3$-dimensional Sasakian space form, that is, isometric to one
  of the Lie groups $SU(2)$ (the Berger sphere),
  $\widetilde{SL(2,\R)}$ (the universal covering of the
  unit tangent bundle of the real hyperbolic space),
  or $Nil^3$ (the Heisenberg group).
  In all cases, the $(+1)$-nullity distribution
  is orthogonal to the contact distribution.
\end{thm}

\begin{cor}
  A complete Riemannian manifold modelled on one of the left-invariant metrics
  listed on Table~\ref{table1} is locally isometric to the
  corresponding model.
  \end{cor}

Recall that a Sasakian space form is a Sasakian manifold
of constant $\varphi$-sectional curvature, where the 
$\varphi$-sectional curvature plays the role accorded
to the holomorphic sectional curvature in K\"ahler geometry.
We refer to~\cite{bookBlair} for a discussion of Sasakian geometry. 
In particular the spaces in Theorem~\ref{kappa1}
have the structure of Lie groups and thus
are homogeneous contact metric manifolds. 
A straightforward computation
using~\cite{milnor}
shows that the admissible metrics with~$1$-conullity~$2$ 
are given as follows. The groups $SU(2)$,
$\widetilde{SL(2,\R)}$, $Nil^3$ are unimodular, 
so there is an orthonormal basis $e_1$, $e_2$, $e_3$, where $e_1$
is tangent to the $1$-nullity and
\begin{equation}\label{ei}
  [e_1,e_2]=\lambda_3e_3,\ [e_2,e_3]=\lambda_1e_1,\ [e_3,e_1]=\lambda_2e_2.
  \end{equation}
By changing $e_1$ to $-e_1$, we may assume $\lambda_1>0$, and
then the possibilities for left-invariant metrics are as follows:
\begin{table}[H]
  $\begin{array}{|c|c|c|c|c|c|c|}
     \hline
     \lambda_1 & \lambda_2 & \lambda _3 & M & \mathrm{scal} &
                                                              \textrm{$\varphi$-sect curv} &\textrm{Condition} \\
     \hline
     \theta + 1/\theta & \theta & 1/\theta& SU(2) & 2 &-1& \theta>0 \\ \hline
             &   &   & SU(2) &  && \theta > 0 \\ \cline{4-4}\cline{7-7}
     2         & \theta  & \theta  & \widetilde{SL(2,\R)} & -2+4\theta &
                                                                                           -3+2\theta& \theta < 0 \\  \cline{4-4}\cline{7-7}
      &  &  & Nil^3 &  && \theta=0 \\
     \hline
   \end{array}$
   \smallskip
   \caption{}\label{table1}
   \end{table}
 According to Perrone~\cite{PerroneThree}, there is an additional,
 non-unimodular Lie group structure on the Sasakian space form of 
$\varphi$-sectional
 curvature $<-3$, that is, the simply-connected
 solvable Lie group with Lie algebra
    \[ [e_1,e_2] =\alpha e_2+2\xi,\ [e_1,\xi]=[e_2,\xi]=0, \]
      where $\alpha\neq0$, $\xi$ is the characteristic vector
      field and spans $\N_1$, and $e_1$, $e_2$, $\xi$ is
      orthonormal, is isometric (but not isomorphic) to 
      $\widetilde{SL(2,\R)}$ if $\alpha^2=-2\theta$.

\subsection{Results with $\kappa=0$}

The $3$-dimensional case of the following theorem is proved
in~\cite[Thm.~3]{aazami-thompson}, and a simple proof in the case of arbitrary
dimension is sketched in~\cite[Remark, p.~1324]{FZ}. For the
convenience of the reader, we provide an alternate, and as well
simple, argument in subsection~\ref{sec:kappa0}. 

\begin{thm}\label{kappa0}
  Let $M$ be a simply-connected complete 
  Riemannian $n$-manifold with
  maximal $0$-conullity~$2$. Assume the scalar
  curvature function~$s$ is positive and bounded away from zero.
  Then $M$ splits as the Riemannian product $\R^{n-2}\times\Sigma$,
  where $\Sigma$ is diffeomorphic to the $2$-sphere. 
\end{thm}

The splitting in Theorem~\ref{kappa0} ceases to be true if $s$ 
attains negative values, as the examples constructed by
Sekigawa~\cite{sekigawa} show. Recently,
a complete description of the metrics on complete simply-connected
locally irreducible $3$-manifolds with constant $0$-nullity $1$ and constant
negative scalar curvature, as well as the topology of
their quotients in case the fundamental group is finitely
generated, has been obtained~\cite{brooks2020}. 
In view of Theorem~\ref{kappa0}
and~\cite[Thm.~1]{brooks2019}, the following question seems interesting:

\begin{quest}
  Is there a simply-connected complete irreducible Riemannian $n$-manifold
  with constant~$0$-conullity~$2$ and nonnegative sectional curvature?
  \end{quest}

In~\cite[Thm.~3]{brooks2019} it was claimed that complete $4$-manifolds
with $0$-nullity~$1$, non-zero splitting tensor,
and finite volume do not exist. In contrast,
we show:

\begin{thm}\label{nul1-kappa0}
  There exists a compact irreducible locally homogeneous Riemannian $5$-manifold
with $0$-nullity~$1$.
\end{thm}

The manifold in Theorem~\ref{nul1-kappa0}
is in fact an almost Abelian
Lie group. The idea of construction follows the example
in~\cite[\S9]{DOV}. Other examples of compact locally homogeneous
spaces with non-trivial nullity are given in~\cite{DOV2}. 

 \subsection{Results with $\kappa=-1$}

 In the $3$-dimensional case,
 the following result is closely related to~\cite[Thm.~1.1]{sw}.
 
\begin{thm}\label{kappa(-1)}
  Let $M$ be a Riemannian $n$-manifold ($n\geq3$) with
  maximal $(-1)$-conullity~$2$. 
  \begin{enumerate}[(a)]
  \item If the scalar curvature is constant and
    $\D=\N_{-1}^\perp$ is integrable on an open subset $U$ of  
$(-1)$-conullity~$2$ then
    $U$ is locally isometric to
    the group of ridig motions of the Minkowski plane,
    $E(1,1)=SO_0(1,1)\ltimes\R^2$,
    with a left-invariant metric.
  \item Assume $M$ is complete and has finite volume. Assume, in addition, 
that either~$n=3$ or the scalar curvature 
bounded away from $-n(n-1)$. Then the universal covering of~$M$ is
    homogeneous. 
    \item If $M$ is homogeneous
    and simply-connected, then $M$ is isometric to $E(1,1)$
    or $\widetilde{SL(2,\R)}$ with a left-invariant metric.
    \end{enumerate}  
  \end{thm}

The left-invariant metrics in Theorem~\ref{kappa(-1)}(c) can also
  be described following~\cite{milnor}. The groups listed are unimodular
  and we use the above notation to write~(\ref{ei}). By switching
  $e_2$ and $e_3$, and changing
  $e_2$ to its opposite, if necessary, we may assume $0<\lambda_3\leq1$, and
  then
  the possibilities are ($K_{\mathcal D}$ denotes the sectional curvature
  of the $2$-plane orthogonal to $\N_{-1}$):
  \begin{table}[H]
    $\begin{array}{|c|c|c|c|c|c|c|}
     \hline
     \lambda_1 & \lambda_2 & \lambda _3 & M & \mathrm{scal} &K_{\mathcal D}&\textrm{Condition} \\
     \hline
     \multirow{2}{*}{$\theta - 1/\theta$} &\multirow{2}{*}{$-1/\theta$}&\multirow{2}{*}{$\theta$}&\widetilde{SL(2,\R)} &\multirow{2}{*}{$-2$} &\multirow{2}{*}{$1$}& 0<\theta<1\\ 
             &   &   & E(1,1) &  && \theta =1 \\ \cline{4-4}\cline{7-7}
     \hline
   \end{array}$
\smallskip
\caption{}\label{table2}
\end{table}

  Note that $SL(2,\R)$ and $E(1,1)$ are both modelled on the same
  algebraic curvature tensor.

The following theorem deals with a situation of least non-trivial nullity. 
We may assume $n\geq4$ as the case $n=3$ is covered by 
Theorem~\ref{kappa(-1)}. 

 \begin{thm}\label{nul1}
   Let $M$ be a complete Riemannian $n$-manifold ($n\geq4$) with constant
 $(-1)$-nullity $1$ and finite volume. Then $n$ is odd and the universal
 Riemannian covering of $M$ is isometric to the solvable
 (almost-Abelian, unimodular) Lie group $G=\R\ltimes\R^m$ ($m=n-1$),
 where $\R^m$ is Abelian, and the adjoint action
 of a certain element of $\R$ on $\R^m$ is given in an orthonormal basis
by the
 matrix $\left(\begin{smallmatrix} I_{m/2}&0\\0&-I_{m/2}\end{smallmatrix}\right)$
 \end{thm}

According to formulae~(\ref{R-almost-abelian}) below, 
if we replace the matrix in the statement of 
Theorem~\ref{nul1} by  $\left(\begin{smallmatrix} I_k&0\\0&-I_{m-k}\end{smallmatrix}\right)$, where $k=1,\ldots,m-1$, 
we get a homogeneous (hence complete) 
 Riemannian $n$-manifold $\R\ltimes\R^m$ with constant
 $(-1)$-nullity $1$, but it will not have quotients of finite volume, as 
it will not be unimodular, unless $k=m/2$.

\subsection{Complete non-integrability of the conullity distribution}

In the last part of this work, we give a simpler and unified
proof of the following results due to Vittone~\cite{vittoneArtTese}
and Di Scala, Olmos and Vittone~\cite{DOV}.

\begin{thm}\label{v-dov}
  Let $M$ be either:
  \begin{enumerate}
  \item[(a)] a connected complete Riemannian manifold of
    with nonzero constant index of $\kappa$-nullity, where $\kappa>0$; or
\item[(b)] a connected simply-connected irreducible homogeneous Riemannian manifold
    with nonzero index of $0$-nullity.
\end{enumerate}
    Then any two points of M can be joined by a
    piecewise smooth curve which is orthogonal to the distribution
    of $\kappa$-nullity at smooth points.
\end{thm}

We wish to thank Wolfgang Ziller for helpful comments. 

\section{Preliminaries}\label{prelim}

The splitting tensor of the nullity distribution
was introduced by Rosenthal in~\cite{RDefSplittingTensor} 
under the name 'conullity operator'; it plays a key role in this work.
Let $M$ be a connected Riemannian manifold, and let $\D$ be a smooth 
distribution on $M$. Consider the orthogonal splitting
$TM=\D\oplus\D^\perp$. It will be convenient to call the $\D$-component
(resp.~$\D^\perp$-component) of tangent vectors the \emph{horizontal}
(resp.~\emph{vertical}) component, and, for a vector field
$X\in\Gamma(TM)$, we shall write $X=X^h+X^v$. 
Now we can define the 
\emph{splitting tensor} of $\D^\perp$ as the map
\[ C:\Gamma(\D^\perp)\times\Gamma(\D)\to \Gamma(\D) \]
given by
\[ C(T,X)=-(\nabla_XT)^h=C_TX \]
(see~\cite[p.~186]{DT}).
It is clear that $C$ is $C^\infty(M)$-linear in each variable. 
Note that in case $\D$ is integrable, $C$ is nothing but the
shape operator of the leaves; further, this is the case
if and only if $C_{T_p}:\D_p\to\D_p$ is a symmetric
endomorphism for all $T\in\Gamma(\D^\perp)$ and all $p\in M$, since
\begin{equation}\label{symm}
\begin{aligned}
  \langle C_TX,Y\rangle - \langle X,C_TY\rangle &=-\langle\nabla_XT,Y\rangle
  +\langle X,\nabla_YT\rangle \\ 
  &=\langle T,\nabla_XY\rangle -\langle \nabla_YX,T\rangle \\ 
                                                & = \langle T, [X,Y] \rangle,
\end{aligned}
\end{equation}
                                                  for all $X$, $Y\in\Gamma(\D^ \perp)$.
                                                  Of course, $C$ vanishes identically if and only if $\D$ is autoparallel.

                                                  In the remainder
                                                  of this section, we assume that $\nu_{\kappa}(p)>0$ for some
                                                  $\kappa\in\R$
and for all $p\in M$,
and we let $\Delta\subset \N_\kappa$ be a nontrivial autoparallel
distribution. In~\cite[Lem., p.~474]{R2}
the following Ricatti-type ODE for the splitting tensor
is used (see also~\cite[Lem.~1]{ferus}).
\begin{prop}
The splitting tensor $C$ of $\Delta$ satisfies
	\begin{equation}\label{eq:spttns}
	\nabla_TC_S=C_SC_T+C_{\nabla_TS}+\kappa\left\langle T,S\right\rangle I
	\end{equation}
for all $S$, $T\in\Gamma(\Delta)$. In particular, 
the operator $C_{\gamma'}$, along a unit speed geodesic $\gamma$ 
in a leaf of $\Delta$, satisfies 
	\begin{equation}\label{eq:stde}
	(C_{\gamma'})'=C_{\gamma'}^2+\kappa I,
	\end{equation}
where the prime denotes covariant differentiation along $\gamma$.
\end{prop}

In general, we shall use the name \emph{$\kappa$-nullity geodesic} to refer
to a geodesic contained in a leaf of $\kappa$-nullity. 

Following the ideas of \cite{trabBoys}, 
we can provide an explicit solution of equation~\eqref{eq:stde}.

\begin{prop}\label{eqn-c}
  Let $\gamma:[0,b) \to M$ be a nontrivial unit speed
  geodesic with $p=\gamma(0)$ and  
$\gamma'(0)\in\Delta_p$ so that $\gamma$ is a geodesic of the 
leaf of $\Delta$ through~$p$. Assume that $\gamma([0,b))$
is contained in an open subset of~$M$ where $\nu_\kappa$ is constant. 
Then the splitting tensor $C_{\gamma'(t)}=C(t)$ of $\Delta$ at $\gamma(t)$  
is given, in a parallel frame along $\gamma$,  by 
\begin{equation}\label{eq:splitting}
	C(t)=-J_0'(t)J_0(t)^{-1},
	\end{equation}
where
\begin{equation}\label{eq:defJacobi}
\begin{aligned}
J_0(t) &=	\begin{cases}
\cos(\sqrt\kappa t) I - \frac{\sin(\sqrt{\kappa}t)}{\sqrt\kappa}C_0 & \text{if $\kappa>0$,} \\
\cosh(\sqrt{-\kappa}t) I - \frac{\sinh(\sqrt{-\kappa}t)}{\sqrt{-\kappa}}C_0 & \text{if $\kappa<0$,} \\
I - tC_0 & \text{if $\kappa=0$,}
\end{cases}
\end{aligned}
\end{equation} 
and $C_0=C(0)$. In particular $J_0(t)$ is invertible for $t\in[0,b)$.
\end{prop}

Recall that the maximum number of linearly independent smooth vector
fields on $S^{m-1}$ is given by $\rho(m)-1$, where $\rho(m)$ is the
\emph{$m$th Radon-Hurwitz number}, defined as $2^c+8d$,
where $m=(\mathrm{odd})2^{c+4d}$ for $d\geq0$ and $0\leq c\leq 3$.  
The invertibility of $J_0(t)$ in Proposition~\ref{eqn-c} implies:

\begin{cor}\label{real-evs}
  Let  $\gamma:[0,b) \to M$ be as in Proposition~\ref{eqn-c}, where $b=\infty$.
\begin{enumerate}[(a)]
\item If $\kappa>0$, then the splitting tensor $C_{\gamma'}$ has no
  real eigenvalues. It follows that  $\rho(n-d)\geq d+1$, where
  $n=\dim M$ and $d=\dim\Delta$ 
\item If $\kappa\leq0$, then any real eigenvalue~$\lambda$ of $C_{\gamma'}$
  satisfies $|\lambda|\leq\sqrt{-\kappa}$.
\end{enumerate}
\end{cor}

\Pf The assertion about the Radon-Hurwitz number in~(a) goes as follows
(cf.~\cite[Thm.~1]{ferus}). Fix an orthonormal basis $T_1,\ldots,T_d$ of
$\Delta$. For every unit $X\in\Delta^\perp$, the list $X$, $C_{T_1}X,\ldots,C_{T_d}X$
must be linearly independent, for otherwise $C_T$ would have a real eigenvalue
for some $T\in\Delta$. Now $C_{T_1}X,\ldots,C_{T_d}X$ projects to a
global frame on the unit sphere of $\Delta^\perp$. \EPf

\medskip

Recall that a smooth distribution $\D$ on $M$ is called
\emph{bracket-generating} if the iterated Lie brackets
of smooth sections of $\D$ eventually span the whole $TM$.
More precisely, we identify $\D$ with its sheaf of smooth local sections, 
and define $\D^1=\D$ and 
\begin{equation}\label{D}
  \D^{r+1} = \D^r + [\D,\D^r] = [\D,\D^r]
  \end{equation}
for $r\geq1$, where
\[ [\D,\D^r] = \{[X,Y]:X\in\D,\ Y\in\D^r\}. \]
Note that $\D^r$ for $r\geq2$ in general has variable rank.
We say that $\D$ is \emph{bracket-generating of step~$r$}
if, for some $r\geq2$, we have $\D^r=TM$ and $r$ is the minimal integer
satisfying this condition. 

\begin{cor}\label{bg2}
  If $M$ is complete,
  $\kappa>0$ and $\nu_\kappa$ is constant,
  then $\D:=\mathcal N_\kappa^\perp$ is bracket-generating of step~$2$.
\end{cor}                                                  

\begin{proof}
  The calculation~(\ref{symm}) shows that $C_{T_p}$ is symmetric for
  $T_p\perp\D^2_p$ and $p\in M$, and thus has all eigenvalues real.
  Now Corollary~\ref{real-evs} implies $T_p=0$. 
\end{proof}

\begin{rmk}
The above
  arguments also easily imply 
  Theorem~4 in~\cite{RukimbiraNull}, which states that
  for a compact $2n+1$-dimensional
  Sasakian manifold $M$ with constant $\nu_\kappa>0$ for some $\kappa>0$
  either $\nu_\kappa\leq n$ or $M$ has constant
  curvature $\kappa$. Indeed, assume $M$ has not constant curvature
  and apply Corollary~\ref{real-evs}(a)
  to it to obtain 
  \[ \rho(2n+1-\nu_\kappa) \geq \nu_\kappa + 1. \]
  Now the desired result immediately follows from
  the trivial estimate $m\geq\rho(m)$ for all $m$.
\end{rmk}

The case $\kappa=0$ of~(\ref{scal2}) below can be found 
in~\cite[ch.~4, \S4.1]{brooks}. 

  \begin{lem}\label{lem:TraceSplitting}
    Let $\gamma: [0,b) \rightarrow M$ with $\gamma'=T\in\Delta$ be as
    in Proposition~\ref{eqn-c}. 
    Denote the scalar curvature of $M$ by $\mathrm{scal}$, 
and put $n=\dim M$ and
    $d=\dim\Delta$.   Then  
    \begin{equation}\label{scal}
 \frac12\frac{d}{dt}\mathrm{scal} = -\kappa (n-d-1) \mathrm{tr}\,C_T +
      \sum_{i\neq j} \langle R(C_T X_i, X_j)X_j,X_i\rangle,
\end{equation}
    where $\{X_i\}_{i=1}^{n-d}$ is a parallel orthonormal frame of $\Delta^{\perp}$
    along $\gamma$.
    
In particular, in case $\Delta=\mathcal N_\kappa$, and $\nu_\kappa=n-2$ along~$\gamma$, we have
\begin{equation}\label{scal2}
 \frac12\frac{d}{dt}\mathrm{scal} = \mathrm{tr}\,C_T (K_{\mathcal D}-\kappa), 
\end{equation}
where $K_{\mathcal D}$ denotes the sectional curvature of the $2$-plane
distribution $\mathcal D=\mathcal N^\perp$. 

Further, if in addition $\mathrm{scal}$ is constant,
then $\mathrm{tr}\,C_T = 0$ and
$\det C_T = \kappa$ along~$\gamma$. 
\end{lem}

\begin{proof}
  Let $\{T=T_1,T_2,\ldots,T_d\}$ be a parallel orthonormal frame of
  $\Delta$ along $\gamma$. We compute
  \begin{equation*}\begin{split}
      \mathrm{scal} &= \sum_{i\neq j} \langle R(T_i,T_j)T_j,T_i\rangle
      \\
      &\qquad+ \sum_{i,j} \langle R(X_i,T_j)T_j,X_i\rangle +
      \sum_{i \neq j} \langle R(X_i,X_j)X_j,X_i\rangle.
      \end{split}\end{equation*}
    Since the first two sums on the right-hand side
    are constant along $\gamma$,
    we get
    \begin{align}\label{eqn1}\nonumber
	\frac{d}{dt}\mathrm{scal} & =\sum_{i\neq j}
        \langle \nabla_TR(X_i,X_j)X_j,X_i\rangle \\ \nonumber
                                  &=-\sum_{i\neq j}
      \langle \nabla_{X_i}R(X_j,T)X_j,X_i\rangle +
      \langle \nabla_{X_j}R(T,X_i)X_j,X_i\rangle\\
&=-2\sum_{i\neq j} \langle \nabla_{X_j}R(T,X_i)X_j,X_i\rangle,
\end{align}
where we have used the Bianchi identity and other symmetries of 
$R$.  

Next, since $T\in\mathcal N_{\kappa}$, we can write
   \begin{align*}
 \nabla_{X_j}R(T,X_i)X_j&=-\nabla_{X_j}(\kappa(T\wedge X_i)X_j)-
R(\nabla_{X_j}T,X_i)X_j \\
&\qquad +\kappa(T\wedge\nabla_{X_j}X_i)X_j+
\kappa(T\wedge X_i)\nabla_{X_j}X_j\\
&=\kappa\langle C_TX_j,X_j\rangle X_i+R(C_TX_j,X_i)X_j.
   \end{align*}
Substituting into (\ref{eqn1}) yields~(\ref{scal}).   

In case $d=\nu_\kappa=n-2$ we have 
\begin{align*}
 \langle R(C_TX_1,X_2)X_2,X_1\rangle + \langle R(C_TX_2,X_1)X_1,X_2)\rangle
&= \langle R(X_1,X_2)X_2,C_TX_1\rangle \\
&\qquad+ \langle R(X_2,X_1)X_1,C_TX_2\rangle\\
&= (\mathrm{tr}\,C_T) K_{\mathcal D}, 
\end{align*}
and~(\ref{scal2}) follows. 

Since $\nu_\kappa=n-2$ along~$\gamma$, we have $K_{\mathcal D}\neq\kappa$
along~$\gamma$. Therefore, in case 
$\mathrm{scal}$ is constant, equation~(\ref{scal2}) yields
$\mathrm{tr}\,C_T=0$ along $\gamma$. Finally, take the trace in~(\ref{eq:stde})
and use the characteristic polynomial $C_T^2+(\det C_T)I=0$ to obtain 
$\det C_T=\kappa$ along~$\gamma$. 
\end{proof}

\begin{lem}\label{tr-spl-kappa0}
  Assume $\kappa\leq0$, $\gamma$ is a complete
  $\kappa$-nullity geodesic, $\nu_\kappa=n-2$ and 
  $K_{\mathcal D}$ is bounded away from $\kappa$ along~$\gamma$. Then $\mathrm{tr}\,C(t)=0$ and $\det
  C(t)=\kappa$ for all~$t\in\mathbb R$.
\end{lem}

\begin{proof} Note that $\frac12\mathrm{scal}=K_{\mathcal D}+m\kappa$,
  where $m=\frac{n^2-n}2-1$. 
  Using
equations~(\ref{eq:splitting}), (\ref{scal2}) and Jacobi's formula, we obtain
\begin{equation*}
    \frac{d}{dt}(K_{\mathcal D}-\kappa) = \mathrm{tr}(-J_0'J_0^{-1})(K_{\mathcal D}-\kappa)\\
    =-\frac{\frac{d}{dt}\det J_0}{\det J_0}(K_{\mathcal D}-\kappa).
  \end{equation*}
Integration of this equation yields
\[ K_{\mathcal D}(t)-\kappa=(K_{\mathcal D}(0)-\kappa)|\det J_0(t)|^{-1}. \]
 
Now $\det J_0(t)$ equals
\[ 1-(\mathrm{tr}\,C_0)t+(\det C_0)t^2 \]
if $\kappa=0$, and 
\[ \frac12\left(1+\frac{\det C_0}{\kappa}\right)+ \frac14\left(1-\frac{\det C_0}{\kappa}-
\frac{\mathrm{tr}\,C_0}{\sqrt{-\kappa}}\right)e^{2t}+
\frac14\left(1-\frac{\det C_0}{\kappa}+\frac{\mathrm{tr}\,C_0}{\sqrt{-\kappa}}\right)e^{-2t} \]
if $\kappa<0$. Since $K_{\mathcal D}$ is bounded away from~$\kappa$,
and the initial point is arbitrary along $\gamma$,
the desired result follows.
    \end{proof}

\begin{lem}\label{div}
Let $M$ be a complete Riemannian $n$-manifold of finite volume and
minimal $(-1)$-nullity~$1$. Then $\mathrm{div}\,T=0$, where~$T$ 
is a unit vector field, tangent to the nullity, defined 
in the open set of minimal nullity. 
\end{lem}

\begin{proof} Let $\gamma$ be an integral curve of $T$,
a complete unit speed $(-1)$-nullity geodesic, and put $C(t):=C_{\gamma'(t)}$. 
According to~(\ref{eq:splitting}),
\begin{equation}\label{c(t)-kappa=-1}
 C(t)=(-\sinh t I + \cosh t C_0)(\cosh t I -\sinh t C_0)^{-1}, 
\end{equation}
where $C_0=C(0)$. Now Jacobi's formula yields
\begin{align*} \mathrm{tr}\,C(t)&=-\frac{\frac{d}{dt}\det(\cosh t I- \sinh t C_0)}{\det(\cosh t I- \sinh t C_0)}\\
                              &=-\frac{P(\xi)}{Q(\xi)},
         \end{align*}
where $\xi=\tanh t$ and 
\[ P(\xi) = \sum_{j=0}^m(-1)^j[(m-j)\xi^{j+1}+j\xi^{j-1}]\sigma_j, \]
and
\[ Q(\xi)=\sum_{j=0}^m(-1)^j\xi^j\sigma_j; \]
here $m=n-1$ and $\sigma_j=\sigma_j(C_0)$ denotes the $j$-symmetric function 
of the eigenvalues of $C_0$. Note that $Q$ is nothing but the 
characteristic polynomial of~$C_0$. 

In order to compute the limits of  $\mathrm{tr}\,C(t)$ as $t\mapsto\pm\infty$,
note that if $Q(1)=Q'(1)=\cdots=Q^{(k-1)}(1)=0$ for some 
$k=0,\ldots,m$, then the alternate sums 
\[ \sum_{j=0}^m(-1)^j\sigma_j=\sum_{j=0}^m(-1)^jj\sigma_j=\cdots=
\sum_{j=0}^m(-1)^jj^{k-1}\sigma_j=0. \]
Therefore
\begin{align*}
P^{(k)}(1) & =  \sum_{j=0}^m(-1)^jj(j-1)\cdots(j-k+2)[(m-2k)j+m+k^2-k]\sigma_j\\
& =  \sum_{j=0}^m(-1)^jj(j-1)\cdots(j-k+2)(m-2k)j\sigma_j\\
&=(m-2k)\sum_{j=0}^m(-1)^jj(j-1)\cdots(j-k+2)(j-k+1)\sigma_j\\
&=(m-2k)Q^{(k)}(1),
\end{align*}
and L' H\^opital rule yields
\[ \lim_{t\to+\infty}\mathrm{tr}\,C(t)=\lim_{\xi\to1}-\frac{P(\xi)}{Q(\xi)}
=-(m-2k). \]

In a similar vein, if $-1$ is a root of multiplicity $k$ of the 
polynomial~$Q$, we compute that $\lim_{t\to-\infty}\mathrm{tr}\,C(t)=m-2k$. 

Next, we use the above calculation to prove the following 
claim: the eigenvalues of~$C_0$
are $-1$ and $+1$, each with multiplicity $m/2$. In particular,
$m$ is even and 
the divergence $\mathrm{div}\, T = \mathrm{tr}\,\nabla T = -\mathrm{tr}\,C_T=0$ 
everywhere. 

Suppose the claim is not true at $p\in M$. Let $\gamma:\R\to M$ be a 
nullity geodesic with $\gamma(0)=p$, $C(t):=C_{\gamma'(t)}$. Now
\[ \lim_{t\to+\infty}\mathrm{tr}\,C(t)=-(m-2k_+),\
\lim_{t\to-\infty}\mathrm{tr}\,C(t)=m-2k_-, \]
where $k_{\pm}$ is the multiplicity of~$\pm1$ as an eigenvalue of $C_0$. 
If $k_+<m/2$ then  $\lim_{t\to+\infty}\mathrm{div}\,T = m-2k_+>0$.
Since $k_+$ is an upper semicontinuous function, we have $m-2k_+>0$ 
on a neighborhood of~$p$ in~$M$. Now we can find a (compact) $m$-disk
transversal to $\mathcal N_{-1}$, containing $p$ in its interior, and 
$t_0$, $L>0$ such that $\mathrm{div}\,T|_{\gamma_x(t)}>L$ for all $t\geq t_0$
and $x\in D$; here $\gamma_x$ denotes the nullity geodesic with 
$\gamma_x(0)=x$, $\gamma_x'(0)=T_x$. Put 
\[ U(t) :=\{\gamma_x(s)\;|\;x\in D, s\geq t\},\ v_s:=\mathrm{vol}(U(t_0+s)). \]
Note that $v_s>0$ since $U(t)$ has non-empty interior, and 
$v_s<\infty$ by our assumption. For $0\leq s_1<s_2$ we have 
$U(t_0+s_2)\subset U(t_0+s_1)$ and thus $v_{s_2}\leq v_{s_1}$. 
On the other hand, the Divergence Theorem and the First
Variation of Volume imply
\[ \frac{d}{ds}v_s = \int_{U(t_0+s)}\mathrm{div}\,T > 0, \]
a contradiction. This proves that $k_+\geq m/2$. 

If $k_-<m/2$, we replace $T$ by $-T$, so that $C_0$ is replaced by $-C_0$
and $k_+$ and $k_-$ are interchanged. Now $k_+<m/2$ and the argument above
leads to a contradiction. Hence $k_-\geq m/2$. Since $k_++k_-\leq m$,
we finally deduce that $k_+=k_-=m/2$. 

\end{proof}

\section{Manifolds with $\kappa$-conullity~$2$}

In this section, we obtain results in 
case $M$ has maximal $\kappa$-conullity~$2$.

\subsection{The case $\kappa=1$}

We now prove Theorem~\ref{kappa1}.
Note that, owing to Corollary~\ref{real-evs}(a),  
$2\geq\rho(2)\geq (n-2)+1$, so $n=3$ and $\nu_1=1$.  
                             
For any $T\in\N_1$, $C_T$ is a $2\times2$ real matrix without 
real eigenvalues, again by Corollary~\ref{real-evs}(a), thus 
with a pair of complex conjugate eigenvalues.
Moreover Lemma~\ref{lem:TraceSplitting} says $\mathrm{tr}\,C_T=0$ and
$\det C_T=1$ if $||T||=1$, so that the eigenvalues of $C_T$ must be~$\pm i$. 
Since $C_T^2=-I$, equation~(\ref{eq:stde}) implies that 
$C_{\gamma'}$ is constant along a unit speed nullity geodesic
$\gamma$ with respect to any parallel orthonormal
frame of $\D:=\N_1^\perp$, therefore we can write
$C_T=\left(\begin{smallmatrix}0&-1\\1&0 
  \end{smallmatrix}\right)$ along~$T=\gamma'$,
with respect to a parallel orthonormal frame of $\D$. 

Since $C_T$ is skew-symmetric, the distribution $\D$ is 
non-integrable. A nowhere integrable rank~$2$ 
distribution in a $3$-manifold must be a contact distribution. 
For $X\in\D$, 
we have 
\[ \langle L_TX,T \rangle =\langle -\nabla_XT,T\rangle
=-\frac12 X\cdot||T||^2 = 0. \]
Now the flow of $T$ preserves $\D$, 
so $T$ is the Reeb (or characteristic) 
vector field of $\D$~\cite[\S~3.1]{bookBlair}. Further, $\nabla T=-C_T$
is skew-symmetric, so $T$ is a Killing field. This says $M$ is a 
$K$-contact distribution~\cite[\S~6.2]{bookBlair}, 
which in dimension~$3$ is equivalent to 
Sasakian~\cite[Cor.~6.5]{bookBlair}.

A $3$-dimensional Sasakian manifold with constant scalar curvature
is locally $\varphi$-symmetric~\cite[Thm.~4.1]{watanabe}. Since $M$ is assumed
complete and simply-connected, it is a globally $\varphi$-symmetric 
space~\cite[Thm.~6.2]{takahashi}. By the classification of 
Sasakian globally $\varphi$-symmetric spaces
in dimension~$3$~\cite[Thm.~11]{blair-vanhecke}, 
we finally deduce that $M$ is a Sasakian space form, that is, 
 those listed in the statement
of Theorem~\ref{kappa1}.

\subsection{The case $\kappa=0$}\label{sec:kappa0}
Now we deal with Theorem~\ref{kappa0}. 

Since $s\neq0$ everywhere, the conullity equals~$2$
everywhere. 
For each $p\in M$,
consider the linear map $C_p:\N_0|_p \to M(2,\R)$.
Lemma~\ref{tr-spl-kappa0}
says that $\mathrm{tr}\,C_S=0$ and
$\det C_S=0$ for $S\in\N_0$, so the image of~$C_p$ is at most one-dimensional.
Let $U$ be the set of points $p\in M$ such that $C_p\neq0$.
On $U$ we choose a unit vector field $T\in\N_0$ spanning the
orthogonal complement to $\ker C$ in $\N_0$.
It follows from equation~(\ref{eq:spttns}) that $\nabla_TS\in\ker C$
for all $S\in\ker C$. Therefore $\nabla_TT=0$. 

Since $\det C_T=0$ and the real eigenvalues of $C_T$ can only be zero, due
to Corollary~\ref{real-evs}(b), the endomorphism $C_T$ is nilpotent. 
Now for each $p\in U$ we can find
an orthonormal basis $X_p$, $Y_p$ of~$\D_p=\N_0^\perp|_p$ such that
$C_TX|_p=0$ and $C_TY|_p=a(p)X_p$ for some $a(p)\neq0$; on a connected
component of~$U$, we may assume $a(p)>0$ for all~$p$. 
Also, it follows from
equation~(\ref{eq:stde}) that $X$ and $Y$ can be taken parallel
along a nullity geodesic, and then also the function~$a$ is constant 
along $\gamma$. By passing to the Riemannian universal
covering of $U$, if necessary,
we may define the orthonormal frame $X$, $Y$ of $\D$
globally. Now the Levi-Civit\`a connection satisfies:
\[ \nabla_TT=\nabla_TX=\nabla_TY=0,\]
\[ \nabla_XT=(\nabla_XT)^v\perp T,\
  \nabla_YT=-aX+(\nabla_YT)^v, \]
\[ \nabla_XX=\alpha Y,\ \nabla_YY=\beta X,\ \nabla_XY=-\alpha X,\
  \nabla_YX=-\beta Y+aT, \]
for some smooth functions $\alpha$, $\beta$ on $U$.
We compute that
\[ R(X,Y)X= (X(a)-a\beta)T+(\alpha^2+\beta^2-X(\beta)-Y(\alpha))Y+a(\nabla_XT)^v, \]
and
\[ R(Y,X)Y= -a\alpha T+(\alpha^2+\beta^2-X(\beta)-Y(\alpha))X. \]
From $T\in\N_0$ we deduce that
\begin{equation}\label{eqns-kappa0}
  \alpha=0,\ X(a)=a\beta,\ s=2(X(\beta)-\beta^2).
  \end{equation}
In particular any integral curve $\eta$ of $X$ in $U$ is a geodesic.
By completeness of $M$, the curve~$\eta$ can be extended to a complete geodesic.
We claim that $\eta$ is entirely contained in~$U$. Indeed
the second equation~(\ref{eqns-kappa0}) yields that
\[ \frac d{dt}\log a(\eta(t))=\beta(\eta(t)), \]
and hence
\[ a(\eta(t))=a(\eta(0))e^{\int_0^t\beta(\eta(\xi))\,d\xi}. \]
The third equation in~(\ref{eqns-kappa0}) says that 
$X(\beta)=\frac12s+\beta^2>0$, so
\[ a(\eta(t))\geq a(\eta(0))e^{t\beta(\eta(0))}>0 \]
for $t>0$. In particular $a$ is bounded away from zero along $\eta$ for
positive time. Repeating the argument for negative time yields
that $\eta$ is contained in $U$.
By assumption $s\geq2\delta^2$ for some $\delta>0$, so
$X(\beta)\geq \delta^2+\beta^2$. 
After integration, we can write
\[ \arctan(\delta^{-1} \beta(\eta(t)))\geq\delta t +
  \arctan(\delta^{-1} \beta(\eta(0))) \]
for all $t\in \R$. This is a contradiction, since the right-hand side
is unbounded. Hence $U=\varnothing$, which is to say $C\equiv0$,
and this implies that $M$ splits. 

\subsection{The case $\kappa=-1$}\label{sec:kappa-negative}

In this subsection, we prove Theorem~\ref{kappa(-1)}. 

We will first consider parts~(a) and~(c) of the statement.  
Note that the scalar curvature is constant under the assumptions
there. For each $p\in M$,
consider the linear map $C_p:\N_{-1}|_p \to M(2,\R)$.
Since the scalar curvature is constant, Lemma~\ref{lem:TraceSplitting}
says that $\mathrm{tr}\,C_T=0$ and
$\det C_T=-1$ for unit $T\in\N_{-1}$,
so $C_p$ is injective and its image lies in the
$3$-dimensional subspace of traceless matrices. Moreover, $C_p$ cannot
be onto the subspace of traceless matrices, as this subspace contains
singular matrices. Therefore $\dim\N_{-1}|_p<3$, and hence $n<5$. 

We next rule out the case $n=4$.
By dimensional arguments the image of~$C_p$ meets the subspace of symmetric
endomorphisms of $\D_p$, for each $p\in M$.
Taking the trace of equation~(\ref{eq:spttns}) throughout, we obtain
\[ \mathrm{tr}(C_SC_T)=2\left\langle T,S\right\rangle  \]
for all $S$, $T\in\N_{-1}$.
Now we can find local orthonormal frames
$T_1$, $T_2$ of $\N_{-1}$ and $X$, $Y$ of $\D$
such that $C_{T_1}$ and $C_{T_2}$ are respectively represented by the
matrices
\[ \begin{pmatrix}1&0\\0&-1\end{pmatrix}\ \mbox{and}\ 
  \begin{pmatrix}0&b\\1/b&0\end{pmatrix}, \]
where $b$ is a nowhere zero locally defined smooth function on $M$.
We refer again to equation~(\ref{eq:spttns}) to write
\[ \nabla_{T_2}C_{T_1}=C_{T_1}C_{T_2}+ C_{\nabla_{T_2}T_1}, \]
and identify the endomorphisms with their matrices to obtain 
\[ 0 =  \begin{pmatrix}0&b\\-1/b&0\end{pmatrix}
  +\langle\nabla_{T_2}T_1,T_2\rangle \begin{pmatrix}0&b\\1/b&0\end{pmatrix}, \]
which clearly is impossible. 

Now $n=3$. Let $T\in\N_{-1}$, $||T||=1$. We already know
that~$\mathrm{tr}\,C_T=0$ and
$\det C_T=-1$, so the eigenvalues of $C_T$ are $\pm1$. 
Since $C_T^2=I$, equation~(\ref{eq:stde}) implies that 
$C_{\gamma'}$ is constant along a nullity geodesic
$\gamma$ with respect to any parallel orthonormal
frame of $\D:=\N_1^\perp$. Then we can write
$C_T=\left(\begin{smallmatrix}-1&0\\0&1 
  \end{smallmatrix}\right)$ along~$T=\gamma'$
with respect to a parallel frame of unit vector fields
$\tilde X$, $\tilde Y$ of $\D$ along $\gamma$. Note that this frame
is orthogonal at $p$ if and only if $C_{T_p}$ is a
symmetric endomorphism if and only if $\D$ is integrable
at~$p$.

In any case we have a locally defined 
frame $T$, $\tilde X$, $\tilde Y$, where we put
$f:=-\langle \tilde X,\tilde Y\rangle$,
and we orthonormalize it to get
\begin{align*}
  X&=\frac1{\sqrt{1-f^2}}(\tilde X+f\tilde Y),\\
  Y&=\tilde Y.
\end{align*}
With respect to $X$, $Y$ we have
\[ C_T=\begin{pmatrix}-1&0\\2F&1 
  \end{pmatrix}, \]
where we have set $F:=f/\sqrt{1-f^2}$. Note that~$f$ is constant
along~$\gamma$, so $X$, $Y$ are parallel along $\gamma$ and $T(F)=0$. 
Hence we
can write the Levi-Civit\`a connection as follows:
\begin{equation}\label{lc}
  \begin{gathered}
    \nabla_TT=\nabla_TX=\nabla_TY=0,\ \nabla_XT=X-2FY,\ \nabla_YT=-Y,\\
 \nabla_XX=-T+\alpha Y,\ \nabla_YY=T+\beta X,\ \nabla_XY=2FT-\alpha X,\
  \nabla_YX=-\beta Y,
\end{gathered}
\end{equation}
for some locally defined smooth functions $\alpha$, $\beta$. 
The bracket relations follow:
\begin{equation}\label{brackets-kappa-negative}
  [X,Y]=2FT-\alpha X+\beta Y,\ [T,X]=-X+2FY,\ [T,Y]=Y.
  \end{equation}
Next, the curvature relations
\[ \langle R(X,Y)X,Y\rangle=-K_{\mathcal D}, \]
where $K_{\mathcal D}$ is the sectional curvature of the
plane spanned by $X$, $Y$, and 
\[  \langle R(X,Y)X,T\rangle=\langle R(T,Y)X,Y\rangle=
  \langle R(X,Y)Y,T\rangle=0, \]
yield the equations
\begin{equation}\label{eqns-kappa-negative}
\begin{aligned}
  \alpha &= -\beta F,\\ 
  T(\beta)&=\beta,\\
  Y(F)&=-\beta(1+F^2),\\
  X(\beta)-FY(\beta) &= K_{\mathcal D}-1.
\end{aligned}                      
\end{equation}

With these equations at hand, we can finish the 
proofs of  parts~(a) and~(c).
In view of~(\ref{brackets-kappa-negative}),
$\D$ is integrable on an open set $U$
if and only if $F$ vanishes identically on $U$.
Assume this is the case.
Equations~(\ref{eqns-kappa-negative}) then imply $\alpha=\beta=0$
and $K_{\mathcal D}=1$. Now~(\ref{brackets-kappa-negative})
reduces to
\[   [X,Y]=0,\ [T,X]=-X,\ [T,Y]=Y. \]
In other words, we have a local orthonornal frame of vector fields
whose Lie brackets have constant coefficients in this frame.
Owing to Lie's third fundamental theorem,
(see also~\cite[(1.4)]{griffiths} or~\cite[Lem.~2.5]{yamato}), 
$U$ is locally isometric to a Lie group with left-invariant metric;
in this case, $E(1,1)=SO_0(1,1)\ltimes\R^2$. This proves part~(a).

Assume now $M$ is simply-connected and homogeneous as in~(c). Then
$F$ is constant, and the equations~(\ref{eqns-kappa-negative})
say that, again, $\alpha=\beta=0$ and $K_{\mathcal D}=1$.
Lie's third fundamental theorem yields that $M$ is isometric
to $E(1,1)$ in case $F=0$, and to $\widetilde{SL(2,\R)}$ in case $F\neq0$.
This proves~(c).

Finally, we deal with~(b). We assume $M$ is complete and
has finite volume.
Fix a complete unit speed nullity geodesic $\gamma$ in the open
set of minimal $(-1)$-nullity~$\Omega$. 
If the scalar curvature of $M$ is bounded away from $-n(n-1)$,
then we apply Lemma~\ref{tr-spl-kappa0} to deduce that 
the splitting tensor $C_{\gamma'(t)}=C(t)$ satisfies 
$\mathrm{tr}\,C(t)=0$ and $\det C(t)=-1$ for all~$t\in\mathbb R$,
and the argument follows as in the beginning of the proof
to show that we must have $n=3$. 
Then we can construct the locally defined orthonormal frame
$T$, $X$, $Y$ as above, and we have equations~(\ref{eqns-kappa-negative}).
Note that $M$ has $(-1)$-nullity $1$, so Lemma~\ref{div} says that 
$\mathrm{div}\,T=0$. 

We follow an argument in~\cite{sw}.  
Let $\gamma$ be again a complete
$(-1)$-nullity geodesic in $\Omega$, parameterized
so that $T=\gamma'$ along $\gamma$. 
The second equation in~(\ref{eqns-kappa-negative}) implies
that
\begin{equation}\label{beta}
  \beta(\gamma(t))=\beta(\gamma(0))e^t.
\end{equation}
Let $D$ be a compact $2$-disk transversal to $\gamma$ at $\gamma(0)$.
By our assumption that the volume is finite 
and the Poincar\'e recurrence theorem,
$\gamma$ meets $D$ infinitely many times as $t\mapsto\pm\infty$.
Since $\beta$ is bounded on $D$, equation~(\ref{beta}) says that
$\beta$ vanishes identically along $\gamma$. Since $\gamma$ is any
nullity geodesic in~$\Omega$, 
now $\beta=0$ on~$\Omega$. The third
equation in~(\ref{eqns-kappa-negative}) gives $Y(F)=0$ and we
also have $T(F)=0$. Finally, apply the second bracket relation
in~(\ref{brackets-kappa-negative}) to $F$ to get $TX(F)=-X(F)$.
The same argument using the Poincar\'e recurrence theorem implies
that $X(F)=0$. Now $\beta$ and $F$ are constant on~$\Omega$.
In particular the scalar curvature is constant on the closure 
$\bar\Omega$, namely, equal to $2(2\kappa+K_{\mathcal D})=-2$. 
The complement $M\setminus\Omega$ is a closed set consisting of isotropic 
points of $M$, namely, where all sectional curvatures are $-1$ and hence
the scalar curvature is $-6$. By connectedness of $M$, $\Omega=M$. 
Now $\beta$ and $F$ are constant on $M$ and this 
implies that the universal covering of $M$ is homogeneous, via
Lie's third fundamental theorem. This completes the proof
of Theorem~\ref{kappa(-1)}.

\section{Almost Abelian Lie groups}

An \emph{almost Abelian} Lie group~$G$ is a non-Abelian
(real connected) Lie group whose Lie algebra~$\mathfrak g$
has a codimension one Abelian ideal~$V$ (it is equivalent
to require the existence of a codimension one subalgebra~\cite{BC}). 
Hence we can write its Lie algebra as a semidirect product 
$\mathfrak g=\R\ltimes_A V$, where $V$ is an Abelian ideal and 
the action of $\R$ on $V$ is determined by the adjoint action 
of a fixed generator $\xi\in\R$, which we represent
by an operator $A\in\mathfrak{gl}(V)$, so that $[\xi,X]=AX$
for all $X\in V$. Note that $G=\R\ltimes_{e^A}V$, $G$ is unimodular if and only
if $\mathrm{tr}\,A=0$, and $A\neq0$ as $G$ is non-Abelian.  

First we compute the curvature of an almost Abelian 
Lie group $G$, equipped with a left-invariant Riemannian metric
that makes $\xi$ unit, and $\xi$ and $V$ orthogonal. Koszul's formula for the 
Levi-Civit\`a connection immediately 
yields:
\[ \nabla_\xi\xi=0,\ \nabla_\xi X=A^{sk}X,\ \nabla_X\xi=-A^{sy}X,\
\nabla_XY=\langle A^{sy}X,Y\rangle\xi, \]
for all~$X$, $Y\in V$, where $A=A^{sy}+A^{sk}$ is the decomposition 
of $A$ into its symmetric and skew-symmetric components. 
This easily gives expressions for the curvature tensor and the 
sectional curvatures as follows:
\begin{equation}\label{R-almost-abelian}
  \begin{gathered}
      R(X,Y)Z=-\langle A^{sy}Y,Z\rangle A^{sy}X+\langle A^{sy}X,Z\rangle A^{sy}Y,\\
R(X,Y)\xi=0,\\
R(\xi,X)Y=\langle([A^{sk},A^{sy}]-(A^{sy})^2)X,Y\rangle\xi,\\
R(\xi,X)\xi=([A^{sy},A^{sk}]+(A^{sy})^2)X,
\end{gathered}
\end{equation}
for $X$, $Y$, $Z\in V$. 
In particular, the sectional curvatures are given by:
\begin{equation*}
 K(\xi,X) = -||A^{sy}X||^2-\langle[A^{sy},A^{sk}]X,X\rangle, 
\end{equation*}
and
\[ K(X,Y) = - \det\begin{pmatrix}\langle A^{sy}X,X\rangle &\langle A^{sy}X,Y\rangle \\ \langle A^{sy}Y,X\rangle &\langle A^{sy}Y,Y\rangle \end{pmatrix}, \]
for all $X$, $Y\in V$. 

Let $X_1,\ldots,X_m$ be an orthonormal basis of $V$ consisting
of eigenvectors of $A^{sy}$, with corresponding eigenvalues
$\lambda_1,\ldots,\lambda_m$. 
Now we can express the Ricci curvature
as: 
\begin{gather*}
Ric(\xi,\xi)=-\sum_{j=1}^m\lambda_j^2, 
\quad Ric(\xi,X_i)=0, \\
Ric(X_i,X_i) = -\lambda_i\sum_{j=1}^m\lambda_j,\quad
Ric(X_i,X_j) = (\lambda_i-\lambda_j)\langle A^{sk}X_i,X_j\rangle\
\text{($i\neq j$)}. 
\end{gather*}
Finally, the scalar curvature is
\[ scal= -\sum_{i=1}^m\lambda_i^2 -(\sum_{i=1}^m\lambda_i)^2. \]

\begin{lem}\label{lem:nul-aa}
  If $G=\R\ltimes_{e^A}V$ is not flat, then its~$0$-nullity
  distribution is the left-invariant distribution 
defined by the subspace
\begin{equation}\label{nul-aa}
  \ker A^{sy} \cap (A^{sk})^{-1}(\ker A^{sy}).
  \end{equation}
of~$V$. 
\end{lem}

\Pf We use the notation above and formulae~(\ref{R-almost-abelian}).
Suppose $a_0\xi+\sum_{i=1}^ma_iX_i\in\N_0$ for
some~$a_0$, $a_1,\ldots,a_m\in\R$. The nonflatness assumption
implies that $A^{sy}$ is nonzero, so there is an
index~$j$ such that $\lambda_j\neq0$. Now
\begin{align*}
  0&=R(a_0\xi+\sum_{i=1}^ma_iX_i,X_j)X_j\\
   &=a_0R(\xi,X_j)X_j+\sum_{i\neq j}a_iR(X_i,X_j)X_j\\
   &=-a_0\lambda_j^2\xi-\lambda_j\sum_{i\neq j}a_i\lambda_iX_i.
\end{align*}
We deduce that $a_0=0$ (so that $\N_0\subset V$), and $a_i=0$ for
all~$i\neq j$ with $\lambda_i\neq0$. If there is $k\neq j$ with 
$\lambda_k\neq0$, then we repeat the above argument with~$k$ in place of~$j$
to arrive at~$a_j=0$. If, on the contrary, all $i\neq j$ have $\lambda_i=0$, 
we use the following argument for $X=\sum_i a_iX_i\in\N_0$:
\begin{align*}
0&=R(\xi,X)\xi\\
&=A^{sy}A^{sk}X-A^{sk}A^{sy}X+(A^{sy})^2X\\
&=\langle A^{sk}X,X_j\rangle \lambda_jX_j-a_j\lambda_jA^{sk}X_j + a_j\lambda_j^2X_j.
\end{align*}
Since $A^{sk}X_j\perp X_j$, the second term gives that $a_j=0$ 
or $A^{sk}X_j=0$. In the latter case, the first term vanishes, 
and hence the last term gives again $a_j=0$. In any case $a_j=0$, so
$\N_0\subset\ker A^{sy}$. Finally, if $X\in\N_0$ and $Y\in V$ then 
\[ 0=\langle R(\xi,X)\xi,Y\rangle=\langle A^{sk}X,A^{sy}Y\rangle. \]
Therefore $A^{sk}X\in(\mathrm{im}\,A^{sy})^\perp=\ker A^{sy}$, 
proving that $\N_0$ is contained in~(\ref{nul-aa}). 
The converse inclusion is clear from~(\ref{R-almost-abelian}),
and this finishes the proof. \EPf 

\medskip

We are now prepared to prove Theorem~\ref{nul1-kappa0}.
Consider an almost Abelian group
$G=\R\ltimes_{e^A}V$, where
\[ A= \begin{pmatrix}0&-b&0&-c\\
    b&0&0&0\\
    0&0&-a&0\\
    c&0&0&a\end{pmatrix} \]
with respect to a basis $X_1,\ldots,X_4$ of~$V$. Note that 
$G$ is unimodular. 

  Take the left-invariant metric on $G$ obtained by declaring the
  basis $\xi$, $X_1,\ldots,X_4$ orthonormal. Choose the
  coefficients $a$, $b$, $c$ of the matrix~$A$
  to be all nonzero.  Thanks to Lemma~\ref{lem:nul-aa}, we immediately see
  that the nullity distribution is spanned by $X_2$. Now $G$ has
  $0$-nullity~$1$. 

 Suppose $G$ splits as a Riemannian product. Then the conullity
 splits accordingly. It follows that either one of the factors
 is flat, or both have conullity $2$. In the former case, one of the 
 factors coincides with the $0$-nullity,
 but this contradicts the fact that the
splitting tensor $C_{X_2}\xi=-(\nabla_\xi X_2)^h=-bX_1$ is nonzero. 
Therefore we must be in the latter case.
The factors
 are of $0$-conullity~$2$, therefore they are
 semi-symmetric. We note that 
the left-translations of $G$ are 
isometries, and thus must preserve the factors.
Now the factors are homogeneous; by~\cite[Prop.~5.1]{szabo},
 they must be symmetric,
 and hence $G$ is symmetric, a contradiction. Hence $G$ is irreducible. 
  
Finally, $G$ has quotients of finite volume if, in addition, 
$e^A$ can be represented by a matrix with integral coefficients
in some basis of $V$, 
in view of the following result (see~\cite[Cor.~6.4.3]{filip};
  here it is worth mentioning that every finite volume quotient of
  a solvable Lie group by a discrete subgroup is automatically
  compact, a result due to Mostow~\cite{mostow}). 

\begin{lem}[Filipkiewicz's criterion]\label{lem:filip}
  Suppose $G=\R\ltimes_{e^A}V$ is unimodular and non-nilpotent. Then there
  is a discrete subgroup $\Gamma$ of $G$ with $\Gamma\backslash G$ compact
  if and only if there exists $\lambda\in\R$, $\lambda\neq0$, such that
  $e^{\lambda A}$ has a characteristic polynomial with integral coefficients.
  \end{lem}
  
In order to find $A$ satisfying the condition of Lemma~\ref{lem:filip},
consider the standard $4\times4$ matrix with two real eigenvalues 
and two complex conjugate ones:
\[ B=\begin{pmatrix}-\gamma&0&0&0\\0&\gamma-2\alpha&0&0\\
0&0&\alpha&-\beta\\ 0 & 0 &\beta&\alpha\end{pmatrix}. \]
As claimed in~\cite{harshavardhan} (see also~\cite[Prop.~3.7.2.5]{bock}), 
there exists 
$\alpha$, $\beta$, $\gamma\in\R$ such that 
\[ e^B=\begin{pmatrix}e^{-\gamma}&0&0&0\\0&e^{\gamma-2\alpha}&0&0\\
0&0&e^\alpha\cos\beta&-e^\alpha\sin\beta\\ 0 & 0 &e^\alpha\sin\beta&e^\alpha\cos\beta\end{pmatrix} \]
is conjugate to
\[ C=\begin{pmatrix}1&0&0&1\\1&2&0&2\\
0&1&3&0\\ 0 & 0 &1&0\end{pmatrix}. \]
The approximate values are 
\[ \alpha\approx0.308333405,\ \beta\approx0.511773474,\ \gamma\approx1.861109547. \]
Now it suffices to find $a$, $b$, $c\in\R\setminus\{0\}$ such that 
$A$ and $B$ are conjugate. The eigenvalues of $B$ are pairwise 
distinct, so it is enough to know that the characteristic polynomials 
$p_A$ and $p_B$ of $A$ and $B$, resp., are equal. 

We have 
\[ p_A(x)=x^4+(-a^2+b^2+c^2)x^2+ac^2x-a^2b^2, \]
and
\[ p_B(x)=x^4+\sigma x^2 + \mu x + \nu,\]
where
\begin{gather*} \sigma = -2\alpha^2+\beta^2-(\alpha-\gamma)^2,\qquad
\mu=2\alpha((\alpha-\gamma)^2+\beta^2), \\
\nu = (\alpha^2+\beta^2)\gamma(2\alpha-\gamma).
\end{gather*} 
Clearly $\mu>0$ and $\nu<0$. It follows that we can solve 
$a^4+\sigma a^2-\mu a + \nu=0$ for real $a>0$. 
Now $a$, $b=\sqrt{-\nu}/a$, $c=\sqrt{\mu/a}$ yields a matrix $A$
such that $p_A=p_B$. This completes the proof of Theorem~\ref{nul1-kappa0}.

\section{Manifolds of $(-1)$-nullity~$1$}

In this section we prove Theorem~\ref{nul1}. 

By Lemma~\ref{div},  $\mathrm{div}\,T=0$
for any unit vector field $T$ in the nullity $\mathcal N_{-1}$. 
It also follows from the proof of that lemma that $m=n-1$ is even 
and the eigenvalues of~$C_T$
are $-1$ and $+1$, each with multiplicity $m/2$. Fix~$p\in M$. 
Since $C_{T_p}$ has real eigenvalues, it is triangularizable over~$\R$,
hence it is triangularizable in an orthonormal basis. Choose an 
orthonormal basis of $\mathcal N_{-1}^\perp|_p$ with respect to which 
the matrix of $C_{T_p}$ is lower triangular. Let $\gamma$ be a 
nullity geodesic with $\gamma'(0)=T_p$. Parallel translate that 
basis to an orthonormal frame 
along~$\gamma$. Equation~(\ref{c(t)-kappa=-1}) 
shows that $C(t)$ is lower triangular in that frame, with 
eigenvalues $\pm1$, each with multiplicity $m/2$, and off diagonal entries
given by polynomials in $e^t$, $e^{-t}$. The assumption of finite volume
together with the fact that $T$ is divergence-free implies,  
via the Poincar\'e Recurrence
Theorem, that $\gamma$ must come back arbitrarily close to~$p$, 
and infinitely often. We deduce that   
such off diagonal polynomials entries must be constant. Now 
$0=(C_{\gamma'})'=(C_{\gamma'})^2-I$, thanks to~(\ref{eq:stde}).
 In particular $C:=C(t)=C_0$ is conjugate 
to a diagonal matrix with entries $\pm1$, each repeated $m/2$ times. 
Since $C$ is diagonalizable, $\dim\ker(C\pm I)=m/2$, and
so~$C$ has the block form 
\begin{equation}\label{C}
C=\left(\begin{matrix}I_{m/2}&0\\ D&-I_{m/2}\end{matrix}\right). \end{equation} 

Let $T$, $X_1,\ldots,X_m$ be a locally defined orthonormal frame
of $M$, with respect to which $C=C_T$ has the form~(\ref{C}), 
which is parallel along nullity geodesics. Then $\nabla_TT=\nabla_TX_i=0$, 
and $\nabla_{X_i}X_j=\sum_{k=1}^m\Gamma_{ij}^kX_k+\langle C_TX_i,X_j\rangle T$ 
for all $i$, $j=1,\ldots, m$. We claim that all
``Christoffel symbols'' $\Gamma_{ij}^k$ ($i$, $j$, $k=1,\ldots,m$) vanish. 

In order to check the claim, we compute
\[
R(T,X_m)X_i = \sum_k(T(\Gamma_{mi}^k)+\Gamma_{mi}^k)X_k \]
for $i<m$ and 
\[
R(T,X_m)X_m = \sum_k(T(\Gamma_{mm}^k)+\Gamma_{mm}^k)X_k-T. \]
Using that $T\in\N_{-1}$, 
we get $T(\Gamma_{mi}^k)=-\Gamma_{mi}^k$. Poincar\'e recurrence 
now gives that $\Gamma_{mi}^k=0$ for $i$, $k=1,\ldots,m$. Next,
assume by induction that $\Gamma_{ij}^k=0$ for $i=i_0,\ldots,m$,
for some~$i_0$, and for $j$, $k=1,\ldots,m$. Since 
\[ [T,X_{i_0-1}]=C_TX_{i_0-1}=\pm X_{i_0-1}+\mbox{lin.~comb.~of
$X_{i_0},\ldots,X_m$}, \]
we get
\[ R(T,X_{i_0-1})X_j=\nabla_T\nabla_{X_{i_0-1}}X_j\pm\nabla_{X_{i_0-1}}X_j
+\underbrace{\mbox{lin.~comb.~of $\nabla_{X_{i_0}}X_j,\ldots,\nabla_{X_m}X_j$}}_{\text{multiple of T}}, \]
so
\[ 0=\langle R(T,X_{i_0-1})X_j,X_k\rangle = T(\Gamma_{i_0-1,j}^k)\pm
\Gamma_{i_0-1,j}^k, \]
and we use Poincar\'e recurrence again to get $\Gamma_{i_0-1,j}^k=0$. 
This proves that  $\Gamma_{ij}^k=0$  for $i$, $j$, $k=1,\ldots,m$;
in particular, $[X_i,X_j]$ can only have component in $T$, if any. 

Let $f$ be a locally defined smooth function on $M$ representing 
an off diagonal entry of $C$. We already know that $T(f)=0$. 
Now 
\[ TX_m(f)=[T,X_m](f)=C_TX_m(f)=-X_m(f), \]
so $X_m(f)=0$ by Poincar\'e recurrence. Suppose now, by induction, 
that $X_m(f)=\cdots=X_{i_0+1}(f)=0$ for some~$i_0$. Then 
\begin{equation}\begin{split} TX_{i_0}(f)&=[T,X_{i_0}](f)=C_TX_{i_0}(f)\\
&=\pm X_{i_0}(f)+\underbrace{\mbox{(lin.~comb.~of
$X_{i_0+1},\ldots,X_m)(f)$}}_{=0}. 
\end{split}
\end{equation}
Therefore $X_{i_0}(f)=0$, by Poincar\'e recurrence. It follows that~$X_i(f)=0$
for all~$i$ and hence~$f$ is locally constant. 
This already implies that the real vector space spanned by  
the locally defined frame
$T$, $X_1,\ldots,X_m$ is closed under commutators, and hence
$M$ is locally isometric 
to a Lie group with a left-invariant metric, due to Lie's Third Theorem. 

We identify the Lie group. 
Note that, going along directions 
other than $\N_{-1}$, the frame $X_1,\ldots,X_m$ is defined up to 
a transformation $\left(\begin{smallmatrix}P&0\\ 0&Q\end{smallmatrix}\right)$,
where $P\in O(m/2)$ (resp.~$Q\in O(m/2)$) acts on the $+1$-eigenspace 
of $C$ (resp.~$-1$-eigenspace of $C$). On one hand, since we have shown the 
entries of $C$ to be locally constant, the matrix~(\ref{C}) 
has~$D$ constant. On the other hand, (\ref{C}) satisfies
\[ \left(\begin{matrix}P(t)&0\\0&Q(t)\end{matrix}\right)\left(\begin{matrix}I&0\\D&-I\end{matrix}\right)\left(\begin{matrix}P(t)&0\\0&Q(t)\end{matrix}\right)^{-1}=\left(\begin{matrix}I&0\\D&-I\end{matrix}\right) \]
for all smooth curves $P(t)$, $Q(t)\in SO(m/2)$, with $P(0)=Q(0)=I_{m/2}$,
along a smooth curve $c:(-\epsilon,\epsilon)\to M$ with $c(0)=p$, 
$c'(t)\not\in\N_{-1}|_{\gamma(t)}$. 
Hence $D=0$, as $m/2\geq2$. 
This, together with the vanishing of the Christofell 
symbols $\Gamma_{ij}^k$ for $i$, $j$, $k=1,\ldots,m$, means 
the Lie group is $G=\R\ltimes_{e^C}\R^m$, where 
$C=\left(\begin{smallmatrix}I_{m/2}&0\\0&-I_{m/2}\end{smallmatrix}\right)$. 

Since $M$ is complete, its universal covering is isometric to $G$.
Finally, we show that~$G$
indeed admits quotients of finite volume, using
Filipkiewicz's criterion (Lemma~\ref{lem:filip}). 
In fact 
\[ \exp\left(\left(\log\frac{3+\sqrt5}2\right)C\right)=\frac12\begin{pmatrix}(3+\sqrt5)I_{m/2}&0\\0&(3-\sqrt5)I_{m/2}\end{pmatrix} \]
is conjugate to 
\[ \begin{pmatrix}3&& & -1&&\\ 
                  &\ddots&  & &\ddots & \\
                  && 3  & && -1 \\
                  1&& & 0&&\\ 
                  &\ddots&  & &\ddots & \\
                  && 1  & && 0 \end{pmatrix}. \]
This finishes the proof
of Theorem~\ref{nul1}.

\section{Bracket-generation of the $\kappa$-conullity distribution}

Let $M$ be a connected complete Riemannian manifold
with nontrivial $\kappa$-nullity distribution $\N_\kappa$
of constant rank, for some fixed $\kappa\geq0$. 
In this section, we revisit some results related to the
following question: When can two points in $M$ be joined by a
piecewise smooth curve always orthogonal to the $\kappa$-nullity
distribution? An answer is given by Theorem~\ref{v-dov}, which we
now prove.

\subsection{The case $\kappa>0$}

This is part~(a) of the theorem and the answer is easy. By Corollary~\ref{bg2},
$\D=\N^\perp_\kappa$ is bracket generating of step~$2$, and hence,
owing to the Chow-Rashevskii theorem~\cite[Thm.~2.1.2]{montgomeryBook}, any two points
in $M$ can be joined by a piecewise smooth curve which is tangent
to $\D$ at smooth points. 

\subsection{The case $\kappa=0$}

Next we deal with part~(b). 
We put $\D=\N^\perp_0$ and
recall the distributions $\D^r$ introduced in~(\ref{D}).
We also set $\mathcal E^r=(\D^r)^{\perp}$ for $r\geq1$. 

\begin{lem}\label{e2}
  Let $T\in\mathcal E^2$, $X\in\D$ and $Y\in TM$. Then
  $\langle \nabla_YT,X\rangle = 0 $.
\end{lem}

\begin{proof}
  Since $T\perp\D^2$, the calculation~(\ref{symm}) says that 
  $C_T$ is a symmetric endomorphism of $\D$ and hence all of its eigenvalues
  are real. Now  Corollary~\ref{real-evs} implies that $C_T\equiv0$.
  Finally, we decompose $Y=Y^h+Y^v$ according to
  $TM=\D\oplus\N_0$ and recall that $\N_0$ is totally geodesic to obtain
  \[ \langle \nabla_YT,X\rangle = -\langle C_TY^h,X\rangle +
    \langle \nabla_{Y^v}T,X\rangle = 0, \] 
  as wished.
  \end{proof}

\begin{lem}\label{er}
 Let $T\in\mathcal E^{r+1}$, $Y\in\D^r$ and $X\in\D$, for some $r\geq1$. Then
  $\langle \nabla_XT,Y\rangle = 0 $.
\end{lem}

\begin{proof}
We compute
\begin{align*}
\langle \nabla_XT,Y \rangle & = -\langle T,\nabla_XY\rangle \\
& = -\langle T,\nabla_XY\rangle -\langle \nabla_YT,X\rangle\qquad \text{(by Lemma~\ref{e2})} \\
& = \langle T,[Y,X]\rangle \\
&= 0 \qquad\text{(since $[Y,X]\in\mathcal D^{r+1}$),} 
\end{align*}
as desired.
\end{proof}

For a point $p\in M$, consider the integers $n_i(p)=\dim\D^i|_p$ and note that 
the non-decreasing sequence $n_1(p)$, $n_2(p),\ldots$ obviously stabilizes,
say at $r=r(p)$. The \emph{growth vector} 
of $\D$ at $p\in M$ is the integer list 
$(n_1(p),\ldots,n_r(p))$. The distribution $\D$ is called
\emph{regular} at $p$ if the growth vector is locally constant at $p$. 
The subset $M_{reg}$ of regular points for $\D$ is open and dense, and consists
precisely of the points of $M$ where all the $\D^i$'s are 
genuine distributions. Since the growth vector is a lower semicontinuous
function, $M_{reg}$ in particular contains the open set where the 
growth vector is maximal.

Let $\gamma$ be a nullity geodesic. Since $\D$ is parallel along $\gamma$ 
and the parallel transport along $\gamma$ is an isometry, we deduce that 
the growth vector is constant along $\gamma$. It follows that the 
connected components of $M_{reg}$ are foliated by the leaves of $\N_0$. 
Fix such a component, say $U$, with growth vector $(n_1,\ldots,n_r)$.

Note that $\mathcal E^{n_r}\neq0$ if and only if 
$\D$ is not bracket-generating on $U$. 
  By the argument above, $\mathcal E^{n_r}$ is parallel
along a nullity geodesic. 
Moreover $\mathcal E^{n_r+1}=\mathcal E^{n_r}$ so, 
owing to Lemma~\ref{er}, $\nabla_XT\in\mathcal E^{n_r}$
for all $X\in\D$ and $T\in\mathcal E^{n_r}$. 
We have shown that $\mathcal E^{n_r}$ is a parallel
distribution in $M$. Note that the leaves of $\mathcal E^{n_r}$ are 
isometric to a flat Euclidean space $\R^s$ for some $s\geq0$ 
($s=0$ corresponds to the case in which $\D$ is bracket-generating in~$U$). 
By the de Rham decomposition theorem and the
Chow-Rashevskii theorem~\cite[Thm.~2.1.2]{montgomeryBook},
and shrinking~$U$ if necessary, 
$U$ splits as a Riemannian product $U_0\times\R^s$, where $s\geq0$
(compare~\cite[Prop.~5.2 of Ch.~IV]{KN}), and 
any two points of $U_0\times\{0\}$ can be joined by a piecewise smooth curve 
in~$U_0\times\{0\}$ tangent to $\D|_{U_0\times\{0\}}$ at smooth points.

If $M$ is homogeneous and simply-connected then $U=M$.
The irreducibility of $M$ implies $s=0$. This proves~(b)
and finishes the proof of the theorem. 

\newcommand{\etalchar}[1]{$^{#1}$}
\providecommand{\bysame}{\leavevmode\hbox to3em{\hrulefill}\thinspace}
\providecommand{\MR}{\relax\ifhmode\unskip\space\fi MR }
\providecommand{\MRhref}[2]{%
  \href{http://www.ams.org/mathscinet-getitem?mr=#1}{#2}
}
\providecommand{\href}[2]{#2}

                                                  \end{document}